\documentclass[10pt,twoside,onecolumn, leqno]{amsart}

\usepackage[english]{babel}
\usepackage{hyperref}
\usepackage{amscd}
\usepackage{amssymb}
\usepackage{amsfonts}
\usepackage{mathrsfs}
\usepackage{stmaryrd}
\usepackage{amsfonts}
\usepackage{amsmath}
\usepackage{pb-diagram}
\usepackage{comment}
\usepackage{color}
\usepackage{latexsym}
\usepackage{amsxtra}
\usepackage{graphicx}

\newtheorem{theorem}{Theorem}[section]
\newtheorem{proposition}{Proposition}[section]
\newtheorem{lemma}{Lemma}[section]
\newtheorem{remark}{Remark}[section]

\newcommand{\tore}{\mathbb{T}}
\newcommand{\R}{\mathbb{R}}
\newcommand{\T}{\mathbb{T}}
\newcommand{\N}{\mathbb{N}}
\newcommand{\Z}{\mathbb{Z}}

\newcommand\D{\partial}

\newcommand{\hs}{H_{s}}
\newcommand{\hk}{H_{k}}
\newcommand{\hd}{H_{2}}
\newcommand{\htr}{H_{3}}
\newcommand{\hq}{H_{4}}
\newcommand{\hc}{H_{5}}
\newcommand{\hm}{H_{m}}

\newcommand{\hsp}{H_{s'}}
\newcommand{\hspru}{H_{s' + 1}}
\newcommand{\hsprd}{H_{s' +2}}

\newcommand{\lihd}{{L^{\infty}(0, T ;  H_2)}}
\newcommand{\lihtr}{{L^{\infty}(0, T;  H_3)}}
\newcommand{\lihsp}{{L^{\infty}(0, T;  H_{s'} )}}

\newcommand{\omo}{\omega_0}
\newcommand{\omd}{\omega_{\delta}}
\newcommand{\ombd}{\omega^{\beta}_{\delta}}
\newcommand{\omad}{\omega_{\delta}^{\alpha}}
\newcommand{\omod}{\omega_{0, \delta}}

\newcommand{\ua}{u^{\alpha}}
\newcommand{\ub}{u^{\beta}}
\newcommand{\ud}{u_{\delta}}

\newcommand{\uad}{u_{\delta}^{\alpha}}
\newcommand{\uab}{u^{\alpha, \beta}}

\newcommand{\uabn}{u^{\alpha, \beta, \nu}}

\newcommand{\uo}{u_{0}}
\newcommand{\uoa}{u_{0}^\alpha}

\newcommand{\uob}{u_{0}^\beta}
\newcommand{\uod}{u_{0,\delta}}
\newcommand{\uobd}{u_{0,\delta}^{\beta}}
\newcommand{\uok}{u_{0, k}}
\newcommand{\uobk}{u_{0, k}^\beta}

\begin{document}

\title[On Euler-Voigt and Navier-Stokes-Voigt models]{Some results on
  the well-posedness of Euler-Voigt and Navier-Stokes-Voigt models.}

\author[L. C. Berselli]{Luigi C. Berselli}

\address{(Luigi C. Berselli) Dipartimento di Matematica Applicata ``U.~Dini,''
  Universit\`a di Pisa, Via F.~Buonarroti 1/c, I-56127, Pisa, Italia}

\email{\href{mailto:berselli@dma.unipi.it}{berselli@dma.unipi.it}}

\urladdr{\url{http://users.dma.unipi.it/berselli}}

\author[L. Bisconti]{Luca Bisconti}

\address{(Luca Bisconti) Dipartimento di Matematica Applicata ``G.~Sansone,''
  Universit\`a di Firenze, Via S.~Marta~3, I-50139, Firenze, Italia}

\email{\href{mailto:luca.bisconti@unifi.it}{luca.bisconti@unifi.it}}

\date{\today}

\begin{abstract}
  We consider the Euler-Voigt equations and the Navier-Stokes-Voigt
  equations, which are obtained by an inviscid $\alpha$-regularization
  from the corresponding equations. The main result we show is the
  structural stability of the system in term of the variations of both
  viscosity of regularization parameters.
\end{abstract}
\maketitle
\section{Introduction}
One of the most challenging problems in scientific computing is that
of producing reliable simulations of turbulent flows. Since the work
of Kolmogorov in 1941, it is well assessed that there are prohibitive
limitations due to the smallest scales which are persistent in flows
at very high Reynolds number. The limitations (due to speed but also
to memory capacity  of the available most powerful computers) in
performing direct numerical simulations makes the field particularly
challenging. Incompressible fluid with constant density are described
by the Navier-Stokes equations
\begin{subequations}
  \begin{align}
    \D_t v + (v \cdot \nabla)\,v - \nu\Delta v+ \nabla p = f,
    \label{eq:NSa}
    \\
    \nabla \cdot v = 0, \label{eq:NSb}
  \end{align}
\end{subequations}
supplemented with initial and boundary conditions, where
$v(t,x)=(v_1,v_2,v_3)$ is the velocity field, $p(t,x)$ denotes the
pressure, $f(t,x)=(f_1,f_2,f_3)$ is the external force, and $\nu>0$ is
the kinematic viscosity. In the sequel we will consider mainly the
problem in the space periodic setting.\par
One main idea is that of studying averaged or filtered equations, see
the mathematical overview in~\cite{BIL2006}. Many models for the
numerical simulation of the Large Scales (the only ones which are
effectively computable and of relevance for the applications) has been
proposed. Among these models in the recent years there has been a lot
of activity, from both the pure and applied mathematicians, around the
so called ``$\alpha$-models.'' These models are based on a
filtering/smoothing obtained with the application of the inverse of
the Helmholtz operator
\begin{equation*}
  \mathrm{I}-\alpha^2\Delta,
\end{equation*}
and two main questions arise: a) How to describe the (nonlinear)
quadratic term $ (\mathrm{I}-\alpha^2\Delta)^{-1}(v\cdot\nabla)\,v$ in
terms of $(\mathrm{I}-\alpha^2\Delta)^{-1}v$ only (interior closure
problem in LES); b) The role of boundary conditions supplementing the
Helmholtz operator and the derived model. Roughly speaking, the effect
of applying the inverse of the Helmholtz operator is that of getting
two more derivatives of the solution under control.  We also recall
that Leray's~\cite{Ler1934} approach to construct weak solutions of
the Navier-Stokes equations by smoothing just the convective velocity
is based on a very similar idea, with regularization made by
convolution.  There is a big variety in the family of
$\alpha$-approximations to the Navier-Stokes equations and we recall
in alphabetic order, without the claim of being exhaustive, some of
the recent publications about the
subject~\cite{BL2009,CHT2005,CLT2006,CHOT2005,
  DE2006,FHT2002,GL2000,ILT2006,LMNR2010,LL2006a,LN2007b,LST2008}. Further
details can be found in the introduction
of~\cite{CHT2005,OT2007,Reb2008}.  We observe that essentially all the
above methods are based on a sort of regularization by a viscous
smoothing of the equations, which is reflected in better analytical
properties.

A very promising and new approach, recently introduced by Cao,
Lunasin, and Titi~\cite{CLT2006}, is that of the inviscid
regularization coming from the Layton-Lewandowski (or simplified
Bardina) model, which is a zeroth order deconvolution method, when the
viscosity vanishes. Authors in~\cite{CLT2006} observed that setting
the viscosity $\nu=0$ in that model gives the following (with
$u=(\mathrm{I}-\alpha^2\Delta)^{-1}v$) Euler-Voigt model:
\begin{subequations}
  \begin{align}
    \D_t u -\alpha^2 \D_t \Delta u + (u \cdot \nabla)\,u + \nabla p =
    f,
    \label{eq:EVa}
    \\
    \nabla \cdot u = 0.  \label{eq:EVb}
  \end{align}
\end{subequations}
Curios enough, when the viscosity is reintroduced, this turns out to
coincide with a model for visco-elastic fluids studied starting from
the seventies by Oskolkov~\cite{Osk1973,Osk1982} and known as the
Navier-Stokes-Voigt (sometimes written Voight) model:
\begin{subequations}
  \begin{align}
    \D_t u - \alpha^2 \D_t \Delta u + (u\cdot \nabla)\,u - \nu \Delta
    u+\nabla p = f, \label{NSV1a}
    \\
    \nabla \cdot u =0. \label{eq:NSV1b}
  \end{align}
\end{subequations}
The idea of the inviscid regularization is very interesting for two
main reasons: 1) this is not a viscous regularization, and the energy
behavior is respected in a more precise way; 2) This regularization
does not introduce new boundary conditions, and gives chances to the
study of the problem in bounded domain, without the aforementioned
usual difficulties coming in theory of Large Eddy Simulations, when
employed in presence of solid boundaries. Related ideas, based on time
relaxation, has been introduced in~\cite{LN2007b}, while the same
inviscid regularization has been also used in the study of water
waves~\cite{BBM1972}, while recent applications to the
quasi-geostrophic equations are given in~\cite{KT2008}.\par
A detailed account of many properties of Voigt equations has been
given in~\cite{KLT2009,KT2009,LT2010,LiT2010} and, beside giving a
motivation based on the role in scientific computing for the study of
the Navier-Stokes and Euler Voigt models, in this paper we treat
questions more linked with the general theory of partial differential
equations, as explained for instance in~\cite{Bei2009}. In particular,
the main results, which are obtained by using in a simplified setting
the techniques introduced by Beir\~ao da
Veiga~\cite{Bei1988,Bei1993a,Bei1994a} concern the well-posedness of
the equations, which is also relevant to understand the stability of
the equations with respect to small perturbation of the
parameters. The results we prove are not surprising and they are strictly
linked with the similar ones recently proved by Linshiz and
Titi~\cite{LiT2010}. Nevertheless our results show how the systems are
robust and this gives new support for their employment in numerical
computations.

\medskip

\noindent\textbf{Plan of the paper.} In Section~\ref{sec:preliminary}
we introduce the notation and give some remarks on the existence of
smooth solutions. In Section~\ref{sec:convergence-to-Euler} we prove
the sharp limits as the regularization parameter $\alpha$ vanishes,
while in Section~\ref{sec:hadamard} we study both limits as $\nu$ and
$\alpha$ vanish.
\section{Some preliminary results}
\label{sec:preliminary}
We introduce now the notation and give some remarks on the existence
of solutions to the Voigt models.

\subsection{Notation}
In the sequel we will use the customary Lebesgue $L^p$ and Sobolev
spaces $W^{k,p}$ and $H^s:=W^{s,2}$, and for simplicity we do not
distinguish between scalar and vector valued functions. Since we will
mainly work with periodic boundary conditions the spaces are made of
periodic functions. In the Hilbertian case $p=2$ we can easily
characterize the divergence-free spaces by using Fourier Series on the
3D torus: Define $\Omega := [0,2\pi]^3$. We denote by $({e}_1, { e}_2,
{ e}_3)$ the orthonormal basis of $\R^3$, and by $x:=(x_1, x_2, x_3)
\in \R^3$ the standard point in $\R^3$.  Let $\mathbb{T}$ be the torus
defined by $\mathbb{T} := \R^3 / \Z^3$.  We use $\|\cdot\|$ to
denote the $L^{2}(\tore)$ norm and we impose the zero mean condition
$\displaystyle \int_{\Omega}\phi \,dx=0$ on velocity, pressure and
external force.  We define, for an exponent $s\geq0$,
\begin{equation*}
  {H}_{s} := \left\{w : \tore \rightarrow \R^3, \, \,
    w \in H^{s}(\tore)^3,   \quad\nabla\cdot w
    = 0, \quad\int_{\tore}w\,dx = {0} \right\},
\end{equation*}
where $H^{s}(\tore)^3:=\big[H^{s}(\tore)\big]^3$ and if $0\leq s<1$
the condition $\nabla\cdot w=0$ must be understood in a weak
sense. For $w \in H_s$, we can expand the velocity field with Fourier
series
\begin{equation*}
  w (x)=\sum_{k \not=0}\widehat{w}_{k}
  e^{i{k \cdot x}},\quad\text{where }k \text{ is the wave-number,}
\end{equation*}
and the Fourier coefficients are given by
$ \widehat{w}_{k}=\frac{1}{|\tore | }\int_{\tore} w(x)e^{-i{k \cdot
    x}}dx$, %
where $|\tore |$ denotes the Lebesgue measure of $\tore$. If $
k:=|k|=\sqrt{|k_{1}|^{2}+|k_{2}|^{2}+|k_{3}|^2}$, then the ${H}_{s}$
norm is defined by
\begin{equation*}
  \| w \|^{2}_{H_s} = \sum_{k\not=0} |k|^{2s} |\widehat{w }_{k}|^{2},
\end{equation*}
where, as above, $\| w \|_{H_0} := \|w\|$. The inner products
associated to these norms are
\begin{equation*}
  (w, v )_{{H}_s} = \sum_{k\not=0} |k|^{2s}  \widehat{w
  }_{k}\cdot\overline{\widehat{v }_{k}}. 
\end{equation*}
We finally characterize $H_s\subset H^s(\tore)$ as follows:
\begin{equation*}
  {H}_s := \Big \{ w  = \sum_{{k}\not=0} \widehat w_{ k}e^{i{k \cdot x}}:\ \  
  \sum_{{k\not=0}} | {k} |^{2s} |\widehat{w}_{ k}|^{2} < \infty,\
  {k}\cdot \widehat{w}_{k}=0,
  \widehat{w}_{-k}= \overline{\widehat{w}_{k}}\  \Big\}.
\end{equation*}
Moreover, we will use the smoothing operator defined for all
$\delta>0$ and for all functions $f$ periodic and divergence free, as
follows: For any $m\geq0$ and given $f\in H_m$, then the function
$f_\delta$ defined as
\begin{equation}
  \label{eq:filter}
  f_\delta:=\sum_{|k|<1/\delta}  \widehat{f}_k e^{i k\cdot x},
\end{equation}
is such that $f_\delta\in H_s$ for all $s\in \R$ (hence it is
infinitely differentiable) and $f_\delta\to f$ in $H_m$ as
$\delta\to0$. Other properties of this (truncation in wave-number
space) smoothing operator can be easily obtained from the Fourier
series characterization and will be recalled in
Section~\ref{sec:convergence-to-Euler}.\par
As a final remark on function spaces, we will also use (but only in
Section~\ref{sec:time-periodic}) the characterization of divergence
free subspaces of $L^2$ and $H^1$, with vanishing Dirichlet boundary
conditions for which we refer for instance to Constantin and
Foias~\cite{CF1988}.\par
In the sequel (especially to obtain estimates for solutions in $H_s$,
with $s$ non-integer) we will also use some elementary commutator type
estimates as the following lemma concerning the operator $\Lambda^s$,
$s\in \R^+$ (see e.g.~\cite{KP1988,KPV1991,Tay1991}), where
$\Lambda := (-\Delta)^{1/2}$.
\begin{lemma}
  \label{lem:lemmaRe}
  For $s > 0$ and $1 < r \leq \infty$, and for smooth enough $u$ and
  $v$
  \begin{equation*}
    \|\Lambda^s(u v)\|_{L^r}\leq  
    C(\|u\|_{L^{p_1}}\|\Lambda^s v\|_{L^{q_1}}+ 
    \|v\|_{L^{p_2}} \|\Lambda^s u\|_{L^{q_2}} ),
  \end{equation*}
  where $1/r = 1/{p_1} + 1/{q_1}=1/{p_2} + 1/{q_2}$ and $C$ is a
  suitable positive constant.
\end{lemma}
\subsection{Existence results}
Concerning the Euler-Voigt equations we have the following result of
global existence and uniqueness of solutions proved
in~\cite[Thm.~3.1]{LT2010}: Therein the result is proved in two
different ways, one with the contraction principle and the other one
by means of the Galerkin method. The results explain the hyperbolic
nature of the problem and are a starting point also to prove Gevrey
regularity results. Here we give some remarks on one main technical
point and also explain some (non strictly essentials) details on the
external force.
\begin{theorem} {\rm (Larios and Titi~\cite{LT2010})}
  Let $T > 0$ and let $\uoa \in H_m$ and $f \in C(-T,T;W^{m-1,6/5})$,
  with $\nabla\cdot f=0$, for $m \geq 1$. Then, there exists a unique
  solution $\ua$ of the Euler-Voigt
  equations~\eqref{eq:EVa}-\eqref{eq:EVb} which belongs to $C^1[-T, T
  ];H_m)$.  Moreover,
  \begin{equation*}
    \|\ua (t) \|_{H^m} < C(\alpha,
    \|\uo\|_{H^m},\sup_{-T<t<T}\|f(t)\|_{m-1,6/5},T),
  \end{equation*}
  for all $t \in [-T, T ]$.
\end{theorem}
\begin{proof}
  We do not claim any originality in the result, but we just give an
  alternate proof of one main point, emphasizing also the role of the
  pressure. We use some classical techniques employed also in Beir\~ao
  da Veiga~\cite{Bei1988} even if here due to the space-periodicity
  the proof is much simpler. We give the proof only for $m=1$, because
  this is the most important step. The higher regularity can be
  obtained in the same way by a bootstrapping argument.\par
  By using the Galerkin approach with approximate solutions $\ua_m$
  (where $\ua_m$ is made with a finite Fourier expansion) it is
  elementary to show the \textit{a-priori} estimate
  \begin{equation*}
    \ua_m\in L^\infty(-T,T;H_1).
  \end{equation*}
  By using standard tools it immediately follows that when
  $m\to+\infty$ the approximate functions $\ua_m$ converge to $\ua$
  which is the unique solution solutions
  to~\eqref{eq:EVa}-\eqref{eq:EVb} and which belongs to
  $L^\infty(-T,T;H_1)$. One difficulty consists in passing from
  boundedness to continuity with respect to the time variable. This
  can be obtained in an elementary manner as follows: First observe
  that the pressure $p$ satisfies the Poisson equation
  \begin{equation*}
    -\Delta p^\alpha=\nabla\cdot\big[(\ua\cdot\nabla)\,\ua\big].
  \end{equation*}
  Since $\ua\in L^\infty(-T;T;H_1)$ we have that
  $\nabla\cdot[(\ua\cdot\nabla)\,\ua]\in L^\infty(-T,T;W^{-1,3/2})$,
  where $W^{-1,3/2}:=(W^{1,3})'$ and consequently, by using the
  classical regularity theory for the Poisson equation in the
  periodic setting with zero mean value, we have
  \begin{equation*}
    \nabla p^\alpha\in L^\infty(-T,T;L^{3/2}).
  \end{equation*}
  Then, by comparison it follows that
  \begin{equation*}
    \ua_t-\alpha^2 \Delta \ua_t=-\nabla
    p^\alpha-(\ua\cdot\nabla)\,\ua+f\in L^\infty(-T,T;L^{6/5}), 
  \end{equation*}
  and hence by using again the elliptic regularity to the equation
  $(\mathrm{I}-\alpha^2\Delta) \ua_t=F$ and the Sobolev
  embedding we have
  \begin{equation*}
    \ua_t\in L^\infty(-T,T;W^{2,6/5})\hookrightarrow L^\infty(-T,T;H^{1}).
  \end{equation*}
  By standard results it follows that $\ua$ can be identified with a
  function continuous with values in $H^1$. Coming back (with the
  improved regularity on $\ua$) to the estimates on the convective
  term and on the gradient of the pressure they are both now in
  $C(-T,T;L^{6/5})$ and the same argument shows that
  \begin{equation*}
    \ua_t-\alpha^2 \Delta \ua_t=-\nabla p^\alpha-(\ua\cdot\nabla)\,\ua+
    f\in  C(-T,T;L^{6/5}),
  \end{equation*}
  hence that $\ua_t\in C(-T,T;H^1)$. This ends the proof, since
  $\nabla\cdot \ua=0$. The high order regularity can be
  proved following the same approach.
\end{proof}
To conclude, we also recall the well-known results for the
three-dimensional Euler equations
\begin{subequations}
  \begin{align}
    \D_t u + (u \cdot \nabla)\,u + \nabla p = f,
    \label{eq:Ea}
    \\
    \nabla \cdot u = 0.  \label{eq:Eb}
  \end{align}
\end{subequations}
\textit{i.e,} when $\alpha = 0$.  It is well known that if $\uo \in
H_s$, and $f\in L^1(0,\overline{T},H_s)$ with $s > 5/2$, then there
exists a unique solution to these equations in $C([0, T];H_s)\cap
C^1([0, T]; H_{s-1})$ for some finite time $0<T\leq \overline{T}$
(see, e.g., the review in~\cite{MB2002}).  Let us recall that,
contrary to the Euler-Voigt equations for which we know existence of
global smooth unique solutions, it is an outstanding open problem
determining whether smooth solutions exist globally in time or blow-up
in finite time.  In particular, the best know criterion for the 3D
Euler equations is that of Beale-Kato-Majda. In~\cite[Thm.~5.2]{LT2010}
it is obtained an interesting criterion linking the regularity of the
Euler equations, with the dissipation of the Euler-Voigt equations as
$\alpha\to0$. More precisely it is shown that if
\begin{equation*}
  \sup_{t\in[0\overline{T}[}\limsup_{\alpha\to 0^+}\alpha^2\|\nabla
  \ua(t)\|^2>0, 
\end{equation*}
then the Euler equations starting with the same initial datum $u_0$ of
$\{\ua\}_{\alpha>0}$ develop a singularity in $[0,\overline{T}]$. This
result is a by product of a result of convergence as $\alpha\to0^+$ of
the Euler-Voigt equations. Consequently, the behavior as $\alpha$
vanishes is relevant also in view of applications of this model to
computations for the study of the possible blow-up for the Euler
system. In fact in~\cite[Thm.~5.2]{LT2010} it is also proved that the
solutions $\ua$ of the three-dimensional Euler-Voigt equations
converge to the corresponding solutions $u$ of the three-dimensional
Euler equations, with respect to the norm $L^\infty (0,T; L^2(\T))$,
and with initial data $ \uoa= \uo \in H^s(\T)$, for $s > 5/2$. Our
main interest is to study the sharp convergence, that in the same
space of the initial datum. Here, the situation is a little bit
different from the usual ``Navier-Stokes $\to$ Euler limit'', since
the regularity requested on the initial conditions changes in the two
systems: The Euler-Voigt requires one more derivative, in order to
have uniform estimates in terms of $\alpha$.  As a by product of our
results we also treat the behavior as the positive viscosity $\nu$
converges to zero and, when introducing the viscous problem, we also
prove a result on time-periodic solutions.
\subsection{Time periodic solutions}\label{sec:time-periodic}
In this subsection we give some remarks on the proof of existence of
time-periodic solutions. The results presented here are obtained with
well-established techniques introduced in Prodi~\cite{Pro1960a},
taking the chance also of making some observations on the existence
for the Navier-Stokes-Voigt equations. Since the results proved here
hold true also in a bounded domain with Dirichlet boundary conditions,
we use -just in this section- the notation
\begin{equation*}
  \begin{aligned}
    {H}_{1,\sigma} = \left\{w : \Omega \rightarrow \R^3, \, \, w \in
      H^{1}(\Omega)^3, \quad\nabla\cdot w = 0, \ w_{|\partial \Omega}
      = {0} \right\},
  \end{aligned}
\end{equation*}
and $H_{-1}:=(H_{1,\sigma})'$.  Since we are in a case very similar to
the 2D Navier-Stokes equations (for the Naiver-Stokes-Voigt equations
it is easy to prove existence and uniqueness of solutions for all
times), we can work directly on the solutions, looking for a fixed
point argument in the infinite dimensional space $H_1$. For the
Navier-Stokes equations, the difference between 2D and 3D (regarding
time-periodic solutions) are explained in~\cite[Ch.~4]{Lio1969}, where
also the Galerkin method with the Brouwer fixed point is used to
construct approximations to periodic solutions.
\begin{proposition}
\label{prop:time-periodic}
  Let $\Omega\subset \R^3$ be a bounded, smooth, and open set; let
  $f\in L^2(0,T;H_{-1})$. Then, there exists at least a solution to the
  Navier-Stokes-Voigt equations
  \begin{subequations}
    \begin{align}
      \D_t u - \alpha^2 \D_t \Delta u + (u\cdot \nabla)\,u - \nu
      \Delta u+\nabla p = f\qquad\text{in
      }\Omega\times[0,T], \label{eq:NSV2a}
      \\
      \nabla \cdot u =0\qquad\text{in
      }\Omega\times[0,T], \label{eq:NSV2b}
      \\
      u =0\qquad\text{on }\partial
      \Omega\times]0,T], \label{eq:NSV2c}
    \end{align}
  \end{subequations}
  such that $u(0)=u(T)\in {H_{1,\sigma}}$.
\end{proposition}
\begin{remark}
  The same result holds also for $\Omega=\tore$ with periodic boundary
  conditions. Moreover, the uniqueness of the time-periodic solutions
  still represent an open problem, posing the same difficulties as
  those well-known for the 2D Navier-Stokes equations.
\end{remark}
\begin{proof}
  Since the proof is very standard, just a sketch of the proof of
  Proposition~\ref{prop:time-periodic} is presented here, as a remark
  on what can be proved for the Navier-Stokes-Voigt equations. We
  observe that by the same usual methods based on Galerkin approximate
  functions and Aubin-Lions compactness tool we can construct a weak
  solution such that $u\in L^\infty(0,T;H_{1,\sigma})$ with
  \begin{equation*}
    \|u(T)\|^2+\alpha^2\|\nabla u(T)\|^2+\nu\int_0^T\!\|\nabla
    u(s)\|^2ds\leq   \|u_0\|^2+\alpha^2\|\nabla
    u_0\|^2+C\!\int_0^T\!\|f(s)\|^2_{H_{-1}}\,ds. 
  \end{equation*}
  Observe that the proof of this result can be obtained by making use
  of the comparison argument to prove that
  \begin{equation*}
    u_t-\alpha^2\Delta u_t\in L^2(0,T;H_{-1}),
  \end{equation*}
  hence that $u_t\in L^2(0,T;H_{1,\sigma})$ by the standard
  Lax-Milgram lemma set in the space $H_{1,\sigma}$. In particular,
  this proves that the solution $u\in C([0,T];H_{1,\sigma})$.\par
  In order to prove existence of time-periodic solutions we have show
  that if $\|u(0)\|^2 + \alpha^2\|\nabla u (0)\|^2\leq R^2$ for a
  large enough $R>0$, then the same bound holds true at $t=T$. Taking
  the inner product of~\eqref{eq:NSV2a} with $u$, and by using the
  Poincar\`e inequality we get
  \begin{equation*}
    \frac{d}{dt} \big(\|u\|^2  + \alpha^2\|\nabla u\|^2 \big) +
    c_1(\nu,\Omega,\alpha)\big(\|u\|^2  + \alpha^2\|\nabla u\|^2 \big)\leq
    c_2(\nu,\Omega)\|f\|_{H_{-1}}^2. 
  \end{equation*}
  Consequently we have that
  \begin{equation*}
    e^{c_1 T}(\|u(T)\|^2 + \alpha^2\|\nabla u(T)\|^2) 
    \le \big( \| u (0)\|^2 + \alpha^2\|\nabla u (0)\|^2 \big) +
    \underbrace{ c_2\int_0^T \| f (t)\|^2_{H_{-1}}\, dt}_{:=c_3}.
  \end{equation*} 
  Therefore, to conclude it is sufficient to impose
  \begin{equation*}
    R^2 \geq \frac{c_3}{1 -e^{-c_{1} T}},
  \end{equation*}
  to show that the solution satisfies
  \begin{equation*}
    \|u(T)\|^2  + \alpha^2\|\nabla u(T)\|^2 \leq R^2. 
  \end{equation*}
  The proof follows by observing that the ball $B(0,R)\subset
  H_{1,\sigma}$ is a convex set in an Hilbert space, and it is
  therefore compact in the weak topology. Hence, by using Tychonov
  theorem we can argue that there exists a fixed point of the map
  $u_0\to u(T)$, which is then a $T$-periodic solution to the
  Navier-Stokes-Voigt equations.
\end{proof}
\section{Convergence to the solutions of the Euler equations}
\label{sec:convergence-to-Euler}
In this section we prove the main result of the paper, that is a
precise convergence result of smooth solutions of the Euler-Voigt
equations to smooth solution of the Euler equations.\par
We start with the following result which is not optimal since one
derivative is lost in the convergence. The technical reason, which can
be easily understood, is that the $H_m$-estimates for the solution of
the Euler-Voigt equation starting from a datum in $H_m$ are not
independent of $\alpha$. In fact, in this case only the boundedness of
$\alpha^2\|\ua\|_{H_m}^2$ is known. To have estimates independent of
$\alpha$ (and continuity up to $t=0$) one needs to assume more
regularity on the initial datum. On the other hand, convergence in
${H_{m-2}}$ is relatively easier to be obtained and the loss of two
derivatives can be understood from the presence of the term $\Delta
\ua_t$.  We observe that this kind of results are obtained
in~\cite{LT2010}. New technical difficulties arise in our setting,
since loosing one derivative is in some sense the best result when the
initial datum is the same for both Euler and Euler-Voigt
equations.\par
We first prove an intermediate result, since it represents the main
technical point. Later on we will elaborate on the results which can
be obtained when also the initial data can change (especially in terms
of their regularity).
\begin{theorem}
  \label{thm:teo1}
  Let $u$ be the solution to the Euler
  equations~\eqref{eq:Ea}-\eqref{eq:Eb} with initial condition $\uo
  \in H_3$, and let $\ua$ be the solution
  to~\eqref{eq:EVa}-\eqref{eq:EVb} with initial condition $\uoa=
  \uo$. Let $T>0$ be a common time of existence for both $u$ and
  $\ua$, with $u,\,\ua \in C([0,T];H_3)\cap C^1([0,T];H_2)$.  Then,
  for any sequence $\{\alpha_n\}$, with $\alpha_n >0$ and such that
  $\alpha_n \to 0$ as $n\to \infty$, it holds that
  \begin{equation*}
    \sup_{0<t<T}  \|u^{\alpha_n}(t) - u(t)\|_{H_2}
    \to 0,\quad\textrm{as}\,\,
    n \to \infty.
  \end{equation*}
\end{theorem}
\begin{proof}
  The proof uses the smoothing tool introduced
  in~\cite{Bei1993a,Bei1994a} and explained for the Euler equations in
  the periodic setting in~\cite{Bei2010a}, see
  also~\cite{Mas2007}. Let $\uo\in H_3$ and define $\uod$ as
  in~\eqref{eq:filter}, then $\nabla\cdot\uod=0$ and moreover, by
  direct computation,
  \begin{equation*}
    \| \uod    \|_{\htr} \leq C\|\uo\|_{\htr}, \qquad \| \uod \|_{H_4} \leq
    \frac{C}{\delta}\qquad \text{and}\quad\| \uod \|_{H_5} \leq
    \frac{C}{\delta^2}. 
  \end{equation*}
  In addition, for $s$ such that $0 < s < 3$, it also holds that
  $\|\uod - \uo\|_{\hs}\leq C \delta^{3 - s}$.  Let $\ud$ be the
  solution of the Euler equations with initial condition $\uod$, which
  we will call ``regularized Euler equations.'' Then, in the interval
  $[0, T]$ the following relations hold true
  \begin{equation}
    \label{eq:NSV63u}
    \|\ud(t)\|_{\htr}< C\quad\text{and}\quad
    \|\ud(t)\|_{\hm}< \frac{C}{\delta^{m-3}},\quad\text{with}\,\, m>3.
  \end{equation}
  We write
  \begin{equation*}
    \|\ua - u\|_{\hd}\leq \|\ua-\ud\|_{\hd}+\|\ud-u\|_{\hd} =: I+II,  
  \end{equation*}
  and we estimate $\|\ua-\ud\|_{L^{\infty}(0,T; \hd)}$ and
  $\|\ud-u\|_{L^{\infty}(0,T; \hd)}$.

  \bigskip

  \noindent\textit{Estimate\,\, for }$ \| \ua- \ud\|_{\lihd}$: We
  denote by $\omad := \ud -\ua $ the difference between the solution
  $\ua$ of the Euler-Voigt equations (\ref{eq:EVa})-(\ref{eq:EVb}) and
  $\ud$. For simplicity we will use the notation $\omega := \omad$.
  Thus, we get
  \begin{equation*}
    \begin{aligned}
      \D_t\omega - \alpha^2\D_t \Delta \omega + \nabla (p_\delta -
      p^{\alpha}) & = - \alpha^2\D_t \Delta \ud - (\ud \cdot \nabla)\,
      \ud + (\ua \cdot \nabla)\,\ua
      \\
      & = - \alpha^2\D_t \Delta \ud - (\omega \cdot \nabla)\, \ud
      -(\ua\cdot \nabla)\, \omega,
    \end{aligned}
  \end{equation*}
  and by taking the $\hd$-inner product with $\omega$ we obtain
  \begin{equation}
    \label{eq:NSV65x}
    \begin{aligned}
     \frac{1}{2} \frac{d}{dt}(\|\omega\|_{\hd}^2 + \alpha^2
      \|\nabla\omega\|_{\hd}^2 ) \leq & \alpha^2| ( \Delta \D_t \ud,
      \omega)_{\hd} | + |( (\omega\cdot\nabla)\,\ud, \omega )_{\hd}|
      \\
      & + |( ( \ua\cdot\nabla)\,\omega, \omega )_{\hd}|.
    \end{aligned}
  \end{equation} 
  Next, we estimate the first term from the right-hand side. Let us
  recall that $\ud$ is a solution to the (regularized) Euler equations, so it
  follows that
  \begin{equation*}
    \begin{aligned}
      \alpha^2 \big|(\Delta \D_t \ud, \omega )_{\hd} \big| & =
      \alpha^2\big| ( \D_t \ud, \Delta \omega)_{\hd} \big|
      \\
      &\leq \alpha^2 \big|( (\ud\cdot \nabla )\,\ud, \Delta
      \omega)_{\hd}\big| + \alpha^2 \big|
      ( \nabla p, \Delta \omega)_{\hd}
      \big|
      \\
      &= \alpha^2 \big|( \nabla\big[(\ud\cdot \nabla )\big]\,\ud,
      \nabla\omega)_{\hd} \big|,
    \end{aligned}
  \end{equation*}
  due to periodicity and also to the incompressibility constraint. By using
  the regularity of the solution of the Euler equations, we have then
  \begin{equation*}
    \begin{aligned}
      \alpha^2 | ( \nabla\big[(\ud\cdot \nabla )\,\ud\big], \nabla
      \omega)_{\hd} |& \leq C\alpha^2 ( \|\ud\|_{\htr}^2+\|
      \ud\|_{\hd} \|\ud\|_{\hq} ) \|\nabla \omega \|_{\hd}
      \\
      &\leq
      C\alpha^2(\|\ud\|_{\htr}^2+\frac{\|\ud\|_{\hd}}{\delta})
      \|\nabla\omega\|_{\hd}
      \\
      &\leq C\frac{\alpha^2}{\delta}\|\nabla \omega \|_{\hd},
    \end{aligned}
  \end{equation*}
  where we are supposing for simplicity that $0<\delta <1$ (since we
  will main use values of $\delta$ close to zero). By using classical
  estimates on the convective term as in
  Kato~\cite[Eq.~(2.1)-(2.2')]{Kat1972}, we estimate the other terms
  from the right-hand side of~\eqref{eq:NSV65x} as follows:
  \begin{equation*}
    |((\ua\cdot\nabla)\,\omega), \omega )_{\hd}| \leq  \|\ua\|_{\htr}
    \| \omega \|_{\hd}^2\quad
    \text{and}\quad
    |((\omega\cdot\nabla)\,\ud), \omega )_{\hd}| \leq  \|\ud\|_{\htr}
    \| \omega \|_{\hd}^2. 
  \end{equation*}
  Collecting the above estimates, using the bounds for the solution of
  the (regularized) Euler equations, and with Schwarz inequality we
  get
  \begin{align*}
    \frac{d}{dt}( \|\omega\|_{\hd}^2 + \alpha^2 \|\nabla\omega
    \|_{\hd}^2 ) \leq C (\|\omega\|_{\hd}^2 + \alpha^2 \|\nabla
    \omega\|_{\hd}^2 ) + C\frac{\alpha^2}{\delta^2}.
  \end{align*}
  Thus, by using the Gronwall's lemma we infer that
  \begin{equation}
    \label{eq:NSV61bb}
    \begin{aligned}
      \|\ua -\ud \|_{L^\infty(0,T;\hd)}^2 \leq &
      \big(\frac{\alpha^2}{\delta^2}T + \|\uo-\uod \|_{\hd}^2
      + \alpha^2 \|\nabla ( \uo-\uod)\|_{\hd}^2\big)Ce^{CT}
      \\
      \leq & \big(\frac{\alpha^2}{\delta^2}T + \delta^2 + \alpha^2
      \|\uod - \uo\|_{\htr}^2 \big) C_1(T).
    \end{aligned}
  \end{equation}
  \bigskip
  
  \noindent\textit{Estimate for} $\|\ud-u\|_{\lihd}$: Here, we take
  the $\hd$-energy estimate for $\omd:=\ud-u$ and we find (since they
  are both solutions to the Euler equations with different initial data)
  \begin{equation*}
    \begin{aligned}
     \frac{1}{2} \frac{d}{dt} \|\omd\|_{\hd}^2 &\leq |((\omd\cdot\nabla)\,\ud,
      \omd )_{\hd}|+ |((u\cdot\nabla)\,\omd, \omd)_{\hd} |
      \\
      &\leq C(\|u\|_{\htr}+\|\ud\|_{\htr}) \| \omd\|^2_{\hd}.
    \end{aligned}
  \end{equation*}
  that is $\|\omd\|_{\hd} \leq \|\omod \|_{\hd}e^{CT} \leq \delta\,
  C_2(T)$, and finally
  \begin{equation}
    \label{eq:NSV61b}
    \|\ud - u\|_{L^{\infty}(0, T ;  \hd)}\leq \delta \,C_2(T).
  \end{equation}

  \medskip

  As a consequence of the
  estimates~\eqref{eq:NSV61bb}-\eqref{eq:NSV61b} we can conclude that
  \begin{equation*}
    \begin{aligned}
      \|\ua - u \|_{\lihd} \leq C\big(\delta^2 +
      \frac{\alpha^2}{\delta^2}T + \alpha^2 \|\nabla (\uod -
      \uo)\|_{\hd}^2\big)^{1/2}C(T),
    \end{aligned}
  \end{equation*}
  where $C(T) = \max \big\{C_1^{1/2}(T), C_2(T)\big\}$.  Now, by
  taking $\delta=\delta_n$ such that both $\delta_n$ and
  $\frac{\alpha_n}{\delta_n}$ go to zero as $n$ goes to infinity, we
  obtain the required convergence in the $\hd$-norm.  In particular,
  it follows that $\|\ua - u \|_{L^\infty(0,T;\hd)} = O(\sqrt{\alpha})$.
\end{proof}
By using exactly the same arguments one can easily prove, more
generally, the following result in smoother spaces.
\begin{theorem}
  \label{thm:teo1a}
  Let $u$ be the solution to the Euler
  equations~\eqref{eq:Ea}-\eqref{eq:Eb} with initial condition $\uo\in
  H_{m+1}$, with $m$ integer such that $m\geq 2$, and let $\ua$ be a
  solution to the Euler-Voigt equations~\eqref{eq:EVa}-\eqref{eq:EVb}
  with initial condition $\uoa=\uo$. Let $T>0$ be a common time of
  existence for $u$ and $\ua$, with $\ua ,u \in C([0, T ];H_{m+1})\cap
  C^1([0,T];H_m)$.  Then, for any sequence $\{\alpha_n\}$ with
  $\alpha_n >0$ and such that $\alpha_n
  \to 0$ as $n\to \infty$, it holds that
  \begin{equation*}
    \sup_{0<t<T}\|u^{\alpha_n}(t) - u(t)\|_{H_m}\to 0,\qquad\text{as }
    n \to \infty. 
  \end{equation*}
\end{theorem}

\bigskip

We arrive now to the main result of the paper which shows the precise
convergence in terms of the initial data and of the space without
unnatural loss of regularity.
\begin{theorem}
  \label{thm:theorem2}
  Consider the Euler equations~\eqref{eq:Ea}-\eqref{eq:Eb} with
  initial condition $\uo \in H_3$, and let $T>0$ be a finite time of
  existence for the solution $u \in C([0, T] ;H_3)\cap C^1([0,
  T];H_2)$.  Let $\uab$ be a solution to the Euler-Voigt
  equations~\eqref{eq:EVa}-\eqref{eq:EVb}, with initial condition
  $\uob$, such that
  \begin{equation}
    \label{eq:NSV-teo2a}
    \begin{aligned}
      i)&\quad \uob\in H_4,\quad \text{for }\beta>0
      \\
      ii)&\quad \|\uob-\uo\|_{\htr}\to0\quad\text{as}\quad\beta \to 0.
    \end{aligned}
  \end{equation} 
  Then, for any sequence $\{\beta_n\}$ with $\beta_n >0$ and such that
  $\beta_n \to 0$ as $n\to \infty$ there exists $\{\alpha_n\}$ with
  $\alpha_n>0$ and  
  converging to zero such that
  \begin{equation*}
    \sup_{0<t<T}\|u^{\alpha_n, \beta_n}(t) - u(t)\|_{\htr}\to
    0,\qquad\text{as } n \to \infty.
  \end{equation*}
\end{theorem}

\medskip

To prove the Theorem~\ref{thm:theorem2} we need the following
preliminary lemma showing that we can construct the solution
$u^{\alpha_n,\beta_n}$ in a time interval independent of $n\in \N$, and
this would be enough in order to get weak convergence results by using
the classical compactness methods, even if we are interested in strong
convergence.
\begin{lemma}
  \label{lem:lemma1}
  Under the hypotheses of Theorem $\ref{thm:theorem2}$, it follows
  that, for any positive sequence $\{\beta_n\}$ such that $\beta_n \to
  0$, as $n \to \infty$, we can find a positive sequence
  $\{\alpha_n\}$ with $\alpha_n \to 0$ such that $\|u^{\alpha_n,
    \beta_n}\|_{L^\infty(0,T;\htr)}$ results bounded uniformly in
  $n\in \N$.
\end{lemma}
\begin{proof}
  Consider the Euler-Voigt equations~\eqref{eq:EVa}-\eqref{eq:EVb}
  with initial data $\uob$.  Taking the $H_3$-inner product with
  $\uab$, and with the usual inequalities for the convective term we
  obtain
  \begin{equation*}
    \frac{d}{dt}( \|\uab \|_{\htr}^2 + \alpha^2\|\nabla \uab
    \|_{\htr}^2) 
    \leq C\big( \|\uab \|_{\htr}^2 + \alpha^2 \|\nabla \uab
    \|_{\htr}^2 \big)^{\frac{3}{2}}.
  \end{equation*}
  Thus, we deduce that
  \begin{equation*}
    \|\uab(t) \|_{\htr}^2  +   \alpha^2\|\nabla \uab(t) \|_{\htr}^2 \leq    
    \frac{\|\uob \|_{\htr}^2 +  \alpha^2 \|\nabla \uob \|_{\htr}^2}{ 
      \Big[ 1 - Ct\big( \|\uob \|_{\htr}^2 +
      \alpha^2 \|\nabla \uob \|_{\htr}^2 \big)^{1/2}\Big]^2},
  \end{equation*}
  with $t\in [0, T]$.  Thanks to the properties $i)$ and $ii)$
  in~\eqref{eq:NSV-teo2a}, it follows that, letting $\beta \to 0$,
  then $\|\uob\|_{\htr} \to \|\uo\|_{\htr}$. Next, we can choose
  $\alpha \to 0$ such that $\alpha\|\nabla \uob\|_{\htr}$ remains
  bounded (for instance choose $\alpha=O(\|u_0^\beta\|_{H_4}^{-1})$. )
  in this way $\|\uab \|_{L^\infty(0,T;H_3)}$ results to be uniformly
  bounded as well, in a time interval $[0,T]$ independent of $\beta$.
\end{proof}

\bigskip
We can now give the proof of Theorem~\ref{thm:theorem2}.
\begin{proof}[Proof of Theorem~\ref{thm:theorem2}]
  To estimate $\|\uab - u \|_{\lihtr}$, we write
  \begin{equation*}
    \begin{aligned}
      \|\uab-u\|_{\htr} &\leq
      \|\uab-\ub\|_{\htr}+\|\ub-u\|_{\htr}=:I+II,
      \\
      &\leq \| \uab - \ud\|_{\htr} + \|\ud- \ub\|_{\htr}
      +\|\ub-u\|_{\htr}=: I_1 + I_2+II,
    \end{aligned}
  \end{equation*}
  where $\ub$ is the solution of the Euler equations with initial data
  $\uob$ and $\ud$ is the solution of the Euler equations starting
  from the regularized initial datum $\uod$. (Note that,
  Lemma~\ref{lem:lemma1} applies also to $\ub$, and consequently
  $\|\ub\|_{\htr}$ results to be uniformly bounded with respect to
  $\beta$).

  \bigskip

  \noindent\textit{Estimate for }$\|\uab-\ud \|_{\lihtr}$: By setting
  $\omega:=\ud-\uab$, we get 
  \begin{equation*}
    \D_t \omega +\alpha^2\D_t \Delta\omega+\nabla (p_\delta-p^{\alpha, \beta})
    = - \alpha^2\D_t \Delta \ud-(\omega\cdot\nabla)\,\ud
    -(\uab\cdot \nabla)\, \omega.   
  \end{equation*}
  Taking the $\htr$-inner product with $\omega $, and recalling that
  $u^\delta$ is solution to the (regularized) Euler equations, we
  arrive at
  \begin{equation*}
    \begin{aligned}
     \frac{1}{2} \frac{d}{dt}( \|\omega\|_{\htr}^2 + \alpha^2 \| \nabla
      \omega\|_{\htr}^2) &\leq |( ( \omega\cdot\nabla)\,\ud, \omega
      )_{\htr}| + |( ( \uab\cdot\nabla)\,\omega) ,\omega )_{\htr}|
      \\
      & \quad + \alpha^2 | ( \nabla(\ud \cdot \nabla )\,\ud, \nabla
      \omega)_{\htr} |.
    \end{aligned}
  \end{equation*} 
  From~\cite[Eq.~(2.2)]{Kat1972}, we get
  \begin{equation*}
    |((\uab\cdot\nabla)\,\omega, \omega )_{\htr}| \leq
    \|\uab\|_{\htr} \| \omega \|_{\htr}^2. 
  \end{equation*}
  With direct computations we obtain
  \begin{equation*}
    \begin{aligned}
      \alpha^2 | ( \nabla(\ud \cdot \nabla )\,\ud, \nabla
      \omega)_{\htr} | & \leq C\alpha^2 ( \| \ud\|_{\htr} \|\ud
      \|_{\hq} +\|\ud\|_{\hd} \|\ud \|_{\hc} ) \|\nabla \omega
      \|_{\htr}
      \\
      & \leq C \alpha^2 ({\delta}^{-1} +\delta^{-2})\|\nabla \omega
      \|_{\htr}= \alpha^2 \check C(\delta ) \|\nabla \omega \|_{\htr},
    \end{aligned}
  \end{equation*}
  and also
  \begin{equation*}
    |( ( \omega\cdot\nabla) \ud, \omega )_{\htr}|
    \leq \|\ud\|_{\htr} \| \omega \|_{\htr}^2 + \|\ud\|_{\hq} \|
    \omega \|_{L^\infty}\| \omega \|_{\htr}.
  \end{equation*}
  Now, to estimate the second term on the right-hand in $\hsp$, with
  $\frac{3}{2} < s^{\prime} < 2$, we employ usual techniques and
  Lemma~\ref{lem:lemmaRe} to get
  \begin{equation}
    \label{eq:NSV71}
    \begin{aligned}
     \frac{1}{2} \frac{d}{dt}( \|\omega \|_{\hsp}^2 + \alpha^2\|\nabla \omega
      \|_{\hsp}^2) \leq & (\|\ud \|_{\hspru} +
      \|\uab\|_{\hspru}) \| \omega \|^2_{\hsp}
      \\
      & + \alpha^2 \big|( \nabla\big[(\ud\cdot \nabla )\,\ud\big],
      \nabla \omega)_{\hsp} \big|,
    \end{aligned}
  \end{equation}
  and again
  \begin{equation*}
    \begin{aligned}
      \alpha^2 \big|( \nabla\big[(\ud \cdot \nabla )\,\ud\big], \nabla
      \omega)_{\hsp} \big| & \leq C\alpha^2( \|\ud\|_{\hspru}
      \| \ud\|_{\htr} + \|\ud\|_{\hsprd} \|
      \ud\|_{\hd})\|\nabla \omega \|_{\hsp}
      \\
      & \leq C\alpha^2 \big( \|\ud\|_{\hspru} +
      \delta^{1-s^\prime }\big) \|\nabla \omega \|_{\hsp}
      \\
      & \leq C\alpha^2 \big( 1 + {\delta^{1-s^\prime}}\big) =
      \alpha^2\widetilde C(\delta ) \|\nabla \omega \|_{\hsp}.
    \end{aligned}
  \end{equation*}
  Then, it follows
  \begin{align*}
    \frac{d}{dt}( \|\omega \|_{\hsp}^2 + \alpha^2\|\nabla \omega
    \|_{\hsp}^2)
    \leq C (\|\omega\|_{\hsp}^2 + \alpha^2 \|\nabla \omega\|_{\hsp}^2
    ) + \alpha^2 \widetilde C^2(\delta).
  \end{align*}
  Thus, from the above differential inequality we get an estimate for
  $\|\omega \|_{\hsp}^2$, and taking the $L^{\infty}$-norm on $[0,T]$
  we find
  \begin{equation*}
    \|\omega \|_{\lihsp}^2 \leq \big(
    \|\omo \|_{\hsp}^2 + \alpha^2 \|\nabla \omo \|_{\hsp}^2  +
    \alpha^2 \widetilde C^2(\delta) T \big) C(T) .
  \end{equation*}
  Summarizing the previous estimates, we
  have 
  \begin{equation*}
    \begin{aligned}
      \frac{d}{dt}( \|\omega\|_{\htr}^2 + \alpha^2 \| \nabla
      \omega\|_{\htr}^2) \leq C( \|\ud\|_{\htr} + \|\uab\|_{\htr})\|
      \omega \|_{\htr}^2 + \alpha^2 \check C(\delta ) \|\nabla
      \omega\|_{\htr}
      \\
      + \frac{C(T)^{\frac{1}{2}}}{\delta} \big( \|\omo \|_{\hsp}^2 +
      \alpha^2 (\|\omo \|_{\hspru}^2 + \widetilde C^2(\delta)\, T
      )\big)^{\frac{1}{2}} \| \omega \|_{\htr}.
    \end{aligned}
  \end{equation*}
  To handle the term $\|\omo\|_{\hsp}$ from the right-hand side we
  write
  \begin{equation}
    \label{NSV65yy}
    \|\omo\|_{\hsp} = \|\uod - \uob\|_{\hsp}
    \leq  \|\uod-\uobd\|_{\hsp} 
    + \|  \uobd - \uob \|_{\hsp}.
  \end{equation}
  where $\uobd$ is the regularization of the initial datum $\uob$.
  Expanding $\|\uobd-\uod\|_{\hsp}^2$ and $\| \uobd - \uob \|_{\hsp}$
  in terms of their Fourier coefficients, we have that
  \begin{equation*}
    \begin{aligned}
      {a)}&\ \|\uobd - \uod \|_{\hsp}^2 \leq
      C\sum_{1\leq |k|\leq \frac{1}{\delta}} |k|^{2s'} | \uobk - \uok
      |^2
      \leq C \delta^{2(3 -s')} \| \uob - \uo \|_{\htr}^2,
      \\
      {b)}& \ \|\uobd - \uob \|_{\hsp}^2 \leq
      C\sum_{|k| > \frac{1}{\delta}} |k|^{2s'} | \uobk |^2 \leq
      C\delta^{2(3 -s')} \| \uob\|_{\htr}^2.
    \end{aligned}
  \end{equation*}
  Therefore, using~\eqref{NSV65yy} and the above inequalities, we
  obtain
  \begin{equation}
    \label{NSV-y} 
    \|\omo\|_{\hsp}^2 \leq C\delta^{6 - 2s'}(\| \uob - \uo \|_{\htr}^2 + \| \uob\|_{\htr}^2).
  \end{equation}
  Analogously, we get $\|\omo\|_{\hspru}^2 \leq C\delta^{4 -
    2s'}(\| \uob - \uo \|_{\htr}^2 + \| \uob\|_{\htr}^2). $ Then, the
  differential inequality for the $\htr$-norm becomes
  \begin{equation*}
    \begin{aligned}
      \frac{d}{dt}( \|\omega\|_{\htr}^2 + \alpha^2 \| \nabla
      \omega\|_{\htr}^2) \leq & C(\|\ud\|_{\htr} + \|\uab\|_{\htr})\|
      \omega \|_{\htr}^2
      \\
      & + \alpha^2 \check C(\delta) \|\nabla \omega\|_{\htr} +
      \widehat C (T, \alpha, \beta, \delta) \| \omega \|_{\htr},
    \end{aligned}
  \end{equation*} 
  where
  \begin{equation*}
    \begin{aligned}
      \widehat C (T, \alpha, \beta, \delta) := \Big(\delta^{2 - 2s'}
      (\delta^2 + \alpha^2 ) \big(\| \uob - \uo \|_{\htr}^2 +
      \|\uob\|_{\htr}^2\big) + \frac{\alpha^2}{\delta^2} \widetilde
      C^2(\delta ) T )\Big)^{\frac{1}{2}}C^{\frac{1}{2}}(T).
    \end{aligned}
  \end{equation*}
  After some manipulations we get
  \begin{equation*}
    \begin{aligned}
      \frac{d}{dt}( \|\omega\|_{\htr}^2 + \alpha^2 \| \nabla
      \omega\|_{\htr}^2) \leq & C(\|\ud\|_{\htr} +
      \|\uab\|_{\htr} + 1) ( \| \omega \|_{\htr}^2 +
      \alpha^2\|\nabla \omega \|_{\htr}^2)
      \\
      & + \alpha^2 \check C^2(\delta ) + \widehat C^2(T, \alpha,
      \beta, \delta ) ,
    \end{aligned}
  \end{equation*}
  and by the Gronwall's inequality it follows that
  \begin{equation}
    \label{eq:NSV75}
    \begin{aligned}
      \|\ud - \uab\|^2_{\lihtr} \leq & \Big((\alpha^2 \check
      C^2(\delta ) + \widehat C^2(T, \alpha, \beta, \delta ) )T
      \\
      & + \|\omo\|_{\htr}^2 + \alpha^2 \|\nabla
      \omo\|_{\htr}^2\Big)e^{ C(T, \alpha, \beta, \delta)},
      %
    \end{aligned}
  \end{equation}
  where $C(T, \alpha, \beta, \delta) := \int_0^T C(\|\ud (s)\|_{\htr}
  + \|\uab (s)\|_{\htr} + 1)\,ds$.  Now, we have that
  \begin{equation*}
    \begin{aligned}
      \|\nabla \omo\|_{\htr}^2 \leq C\|\nabla ( \uobd -
      \uod)\|_{\htr}^2 + C\|\nabla ( \uobd - \uob )\|_{\htr}^2.
    \end{aligned}
  \end{equation*}
  Then, expanding $\|\nabla ( \uobd- \uod )\|_{\htr}^2$ in terms of
  its Fourier coefficients, we find 
  \begin{equation*}
    \begin{aligned}
      \|\nabla (\uobd - \uod )\|_{\htr}^2 &\leq C\sum_{1\leq |k|\leq
        \frac{1}{\delta}} |k|^8 |
      \uobk - \uok  |^2 \\
      &\leq \frac{C}{\delta^2}\sum_{1\leq |k|\leq \frac{1}{\delta}}
      |k|^6 |\uobk - \uok |^2 \leq \frac{C}{\delta^2} \| \uob - \uo
      \|_{\htr}^2,
    \end{aligned}
  \end{equation*}
  and it follows that
  \begin{equation}
    \label{eq:NSV75aaa}
    \begin{aligned}
      \alpha^2 \|\nabla \omo\|_{\htr}^2 &\leq
      C\frac{\alpha^2}{\delta^2} \| \uob - \uo \|_{\htr}^2 +
      C\alpha^2\| \uobd - \uob\|_{\hq}^2
      \\
      &\leq C\frac{\alpha^2}{\delta^2} \| \uob - \uo \|_{\htr}^2 +
      C\alpha^2\| \uob\|_{\hq}^2.
    \end{aligned}
  \end{equation}
  Hence, by using $ \|\omo\|_{\htr}^2 \leq C\|\uod - \uo\|_{\htr}^2 +
  C\|\uob - \uo\|_{\htr}^2$ and~\eqref{eq:NSV75aaa}, the
  estimate~\eqref{eq:NSV75} becomes 
{
    \begin{equation}
      \label{eq:NSV75a0}
      \begin{aligned}
        \|\ud  -  \uab &\|^2_{\lihtr} \\
        &\leq \Big(\big(\alpha^2 \check C^2(\delta ) + \widehat C^2(T,
        \alpha, \beta, \delta ) \big)T + C\|\uod - \uo\|_{\htr}^2
        \\
        &\hspace{0.4 cm}+ C(1 + \frac{\alpha^2}{\delta^2})\|\uob -
        \uo\|_{\htr}^2 + C\alpha^2\|\uob \|_{\hq}^2\Big)e^{C(T,
          \alpha, \beta, \delta)}.
      \end{aligned}
    \end{equation}
  } \bigskip

  \noindent\textit{Estimate for }$\| \ud - \ub \|_{\lihtr}$:  Writing
  the $\htr$-energy estimate for $\ombd:= \ud - \ub$ and using, for the 
  sake of brevity, $\omega$ instead of $\ombd$ we have that
  \begin{equation*}
    \frac{1}{2}\frac{d}{dt} \|\omega\|_{\htr}^2 
    \leq   |((\omega\cdot\nabla)\,\ud,\omega)_{\htr}|
    +|((\ub\cdot\nabla)\,\omega,\,\omega)_{\htr}|,    
  \end{equation*}
  and by the usual estimates
  \begin{equation}
    \label{eq:NSV76}
    \begin{aligned}
      \frac{d}{dt} \|\omega \|_{\htr}^2 \leq C(\| \ud \|_{\htr} + \|
      \ub\|_{\htr}) \| \omega \|^2_{\htr} + C\| \ud \|_{\hq} \|\omega
      \|_{L^{\infty}} \|\omega \|_{\htr}.
    \end{aligned}
  \end{equation}
  Next, to estimate the second term on the right-hand side
  of~\eqref{eq:NSV76} we use again an $\hsp$-energy inequality, with
  $\frac{3}{2} < s^{\prime} < 2$. Thus, arguing as in the derivation
  of~\eqref{eq:NSV71} we get
  \begin{equation*}
    \frac{d}{dt} \|\omega \|_{\hsp}^2\leq  C(\|\ud\|_{\hspru} + \|\ub\|_{\hspru}) 
    \| \omega\|^2_{\hsp}, 
  \end{equation*}
  and using relation~\eqref{NSV-y}, we infer that
  \begin{equation*}
    \|\omega \|_{\lihsp} \leq \|\omo \|_{\hsp} C(T)  \leq 
    \delta^{3 -s^\prime }(\|\uob\|_{\htr}^2 + \|\uob - \uo\|_{\htr}^2 )^{1/2}C(T).
  \end{equation*}
  Thus, relation ~\eqref{eq:NSV76} becomes
  \begin{equation*}
    \begin{aligned}
      \frac{d}{dt} \|\omega \|_{\htr} \leq C(\|u^{ \delta} \|_{\htr} +
      \|\ub\|_{\htr} +1) \| \omega \|_{\htr} +
      \overset{\curlywedge}{C}(T,\beta) \delta^{2 -s^\prime },
    \end{aligned}
  \end{equation*}
  where
  $\overset{\curlywedge}{C}(T, \beta) := C(T)(\|\uob\|_{\htr} +
  \|\uob - \uo\|_{\htr} )^{1/2}$. Therefore, applying the Gronwall
  lemma, we deduce that
  \begin{equation}
    \label{eq:NSV78}
    \begin{aligned}
           \| \ud - \ub \|_{L^\infty(0,T;\htr)}\leq &
      \big(\|\uod-\uo\|_{\htr}+\| \uob-\uo\|_{\htr}
       \\ 
       &\hspace{2cm} + \delta^{2 -s^\prime }\overset{\curlywedge}{C}(T, \beta)T \big)
      e^{C(T, \beta, \delta)},
    \end{aligned}
  \end{equation}
{ with $C(T, \beta, \delta) := \int_0^T C(\|\ud (s)\|_{\htr}
  + \|\ub (s)\|_{\htr} + 1)\,ds$.}

  \bigskip
  
  \noindent \noindent\textit{Estimate for }$\| \uab - \ub\|_{\lihtr}$:

{Using ~\eqref{eq:NSV75a0} and ~\eqref{eq:NSV78} we
    have that
    \begin{equation}
      \label{eq:NSV79}
      \begin{aligned}
        \|\uab -  \ub \|^2_{\lihtr} 
        &\leq C \big( \|\uab - \ud \|^2_{\lihtr} + \|\ud - \ub
        \|^2_{\lihtr}\big)
        \\
        &\leq C\Big[ \big(\alpha^2 \check C^2 (\delta) + \widehat C^2(T,
        \alpha, \beta, \delta ) \big)T + C\|\uod - \uo\|_{\htr}^2
        \\
        & \hspace{1.4 cm} + C(1+\frac{\alpha^2}{\delta^2})\|\uob -
        \uo\|_{\htr}^2 + \alpha^2\| \uob \|_{\hq}^2
        \\
        & \hspace{1.4 cm} + \delta^{4 - 2 s'}
        \overset{\curlywedge}{C}{}^2(T, \beta) T^2 \Big] e^{2C(T,
          \beta, \delta) + C(T, \alpha, \beta, \delta)}.
      \end{aligned}
    \end{equation}
  }
 
  \bigskip

  \noindent\textit{Estimate for }$\| \ub - u \|_{\lihtr}$: We split
  $II$ as follows
%
  \begin{equation}
    \label{NSV80} 
    II\leq \|\ub-\ud\|_{\htr} + \|\ud-u\|_{\htr}.
  \end{equation}
  Consider the first term on the right hand-side of~\eqref{NSV80}.  It
  follows that the difference $\| \ud - \ub \|_{\lihtr}$ can be
  estimated as done in~\eqref{eq:NSV78}. For the second term on the
  right hand-side of~\eqref{NSV80}, we can use the
  same~\cite[Eq.~(23)]{Mas2007}.  Hence, we actually get
  \begin{equation*}
    \|\ud - u \|_{\lihtr} \leq    (\|\uod - \uo\|_{\htr} + \delta^{2 -s'}T) C(T).
  \end{equation*}
  with $\frac{3}{2}< s' < 2$.  Now, multiplying the right-hand side of
  the latter inequality by $ e^{C (T, \beta, \delta)}$,
  and then adding to~\eqref{eq:NSV78}, we get the estimate
  \begin{equation}
    \label{eq:NSV82}
    \begin{aligned}
      \| \ub - u \|_{\lihtr} \leq & C\Big(\| \uod - \uo\|_{\htr}
      \big(C(T) + C) +
      \| \uob - \uo\|_{\htr}\\
      & + \delta^{2 -s^\prime } \big(C(T) +
      \overset{\curlywedge}{C}(T, \beta) \big)T\Big) e^{C (T,\beta,
        \delta)}.
    \end{aligned}
  \end{equation}


  \bigskip

  Note that, the terms $C(T), \overset{\curlywedge}{C}(T, \beta)$, and
  $C(T,\beta, \delta)$ are bounded in terms of $\delta$ and
  $\beta$. Then, for any positive sequence $\{\beta_n\}$, such that
  $\beta_n \to 0$ as $n \to \infty$, letting $\delta \to 0$, we obtain
  \begin{equation*}
    \| u^{\beta_{n}} - u\|_{\lihtr} \to 0\quad\text{ as } n \to \infty.
  \end{equation*}
  Furthermore, by relation~\eqref{eq:NSV79}, letting $\alpha_n \to 0$,
  and taking a sequence $\delta_n \to 0$, such that $ \|u_{0,
    \delta_n}^{\beta_n} -\uo\|_{\htr}$, $\alpha_n\|
  u_0^{\beta_n}\|_{\hq}$, $\frac{\alpha_n}{\delta_n^2}$,
  $\frac{\alpha_n}{\delta_n}$, and $\frac{\alpha_n}{\delta_n^{s' -1}}$
  go to zero as $n$ goes to infinity, we find
  \begin{equation*}
    \|u^{\alpha_n, \beta_n} - u^{\beta_n} \|_{\lihtr} \to 0
    \quad\text{as } n \to \infty.
  \end{equation*}
  Hence, we have that
  \begin{equation*}
    \|u^{\alpha_n, \beta_n} - u \|_{\lihtr} \to 0\quad \text{as } n \to \infty,
  \end{equation*}
  and the thesis follows.
\end{proof}
\section{Convergence of the solutions of the Navier-Stokes-Voigt
  equations}
\label{sec:hadamard}
Combining the results of the previous section with similar
computations, we study also the behavior of solutions in terms of the
viscosity. We are still set in the space-periodic case and, for simplicity
we assume $f=0$. Next, we state a convergence result for solutions of
the Navier-Stokes-Voigt equations to the corresponding solutions of
the Euler equations.
%
{
  \begin{theorem}
    \label{thm:teo3}
    Consider the Euler equations~\eqref{eq:Ea}-\eqref{eq:Eb} with
    initial condition $\uo \in H_3$, and let $T>0$ be a finite time of
    existence for the solution $u \in C([0, T];H_3)$.  Let $\uabn$ be
    a solution to the Navier-Stokes-Voigt
    equations~\eqref{eq:NSV2a}-\eqref{eq:NSV2c}, with initial datum
    $\uob$, satisfying the properties $i)$ and $ii)$
    in~\eqref{eq:NSV-teo2a}.  Then, for any choice of positive sequences
    $\{\beta_n\}$ and $\{\nu_n\}$, both converging to zero as $n
    \to \infty$, there
    exists a positive sequence $\{\alpha_n\}$ converging to zero as $n
    \to \infty$, such that
    \begin{equation*}
      \sup_{0<t<T}\|u^{\alpha_n, \beta_n, \nu_n}-u\|_{\htr} \to
      0,\quad\text{as } n \to \infty. 
    \end{equation*}
  \end{theorem}
}
\begin{proof}
  In the following, $\uab$ and $\uabn$ will indicate the solutions of
  the Euler-Voigt and the Navier-Stokes-Voigt equations (with the same
  initial datum $\uob$) respectively.  To estimate $\|\uabn - u
  \|_{\lihtr}$, we take into account the following terms
  \begin{equation}
    \begin{aligned}
      \|\uabn - u\|_{\lihtr} \leq & \| \uabn-\uad\|_{\lihtr}
      \\
      &+ \| \uad - u\|_{\lihtr} =: I + II,
    \end{aligned}
  \end{equation}
  where $\uad$ is the solution of the Euler-Voigt equations with
  respect to the regularized initial datum $\uod$. %
  Then, the term $II$ will be split as follows
  \begin{equation*}
    II\leq \|  \uad - \uab\|_{\lihtr} +  \| \uab- u \|_{\lihtr} =: II_1 + II_2.
  \end{equation*}
  The above splitting is probably not the simplest one, but it is the
  most convenient to employ the results proved in the previous
  section.  \bigskip

  \noindent\textit{Estimate for }$\|\uabn-\uad\|_{\lihtr}$:
  Setting $\omega^{\alpha, \beta, \nu}_\delta:= \uad - \uabn$ (as
  usual for the difference $\omega^{\alpha, \beta, \nu}_\delta$, we
  drop the symbols $\alpha, \beta, \delta$, and $\nu$), we get
  \begin{equation*}
    \D_t \omega-\alpha^2\D_t \Delta \omega -\nu \Delta \omega 
    + \nabla (p_\delta^\alpha - p^{\alpha, \nu})
    = - \nu \Delta \uad     -  (\omega \cdot \nabla)\,  \uad 
    -(\uabn\cdot \nabla)\, \omega. 
  \end{equation*}
  Taking the $\htr$-inner product of the above relation with $\omega
  $, with the same inequalities employed in the previous sections we
  arrive at
  %
  \begin{equation}
    \label{eq:NSV85}
    \begin{aligned}
     \frac{1}{2} \frac{d}{dt}( \|\omega\|_{\htr}^2 +& \alpha^2 \| \nabla
      \omega\|_{\htr}^2 )+\frac{\nu}{2} \|\nabla \omega\|_{\htr}^2
      \leq \frac{\nu}{2} \|\nabla \uad\|_{\htr}^2
      \\
      &+\big( ( \|\uad\|_{\htr} + \|\uabn\|_{\htr}) \| \omega
      \|_{\htr} + \|\uad\|_{\hq} \| \omega \|_{L^\infty} \big)\|
      \omega \|_{\htr}.
    \end{aligned}
  \end{equation} 
  In order to estimate $\| \omega \|_{L^\infty (\T^3)}$, we use again
  the same tool (with the $\hsp$-energy inequality, for $\frac{3}{2} <
  s^{\prime} < 2$.) By Lemma~\ref{lem:lemmaRe}, we get
  \begin{equation*}
    \begin{aligned}
      \frac{1}{2}    \frac{d}{dt}( \|\omega \|_{\hsp}^2 + \alpha^2\|\nabla \omega
      \|_{\hsp}^2) & \leq C(\|\uad \|_{\hspru} +
      \|\uabn\|_{\hspru} ) \| \omega \|^2_{\hsp} +
      \frac{\nu}{2} \| \nabla \uad \|_{\hsp}^2
      \\
      & 
      \leq C(\| \omega \|^2_{\hsp} + \alpha^2\|\nabla
      \omega\|_{\hsp}^2) + C\nu .
    \end{aligned}
  \end{equation*}
  Then, it follows that
  \begin{equation*}
    \begin{aligned}
      \|\omega \|_{\lihsp}^2 &\leq \big( (\|\omega_0\|_{\hsp}^2 +
      \alpha^2 \|\nabla \omega_0\|_{\hsp}^2 ) + C\nu T \big) C(T).
    \end{aligned}
  \end{equation*}
  Now, it holds that for $0\leq k\leq 3$
  \begin{equation}
    \label{NSV86a}
    \begin{aligned}
      \|\omega_0 \|_{\hk}^2 = \|\uod - \uob\|_{\hk}^2\leq &
      C\|\uod-\uobd\|_{\hk}^2 + C\|\uobd - \uob\|_{\hk}^2
      \\
      \leq & C\delta^{6 - 2k}(\|\uob - \uo \|_{\htr}^2 + \| \uob
      \|_{\htr}^2).
    \end{aligned}
  \end{equation}
  In particular, for $k=s'$ and $k=s'+1$ we get
  \begin{equation*}
    \begin{aligned}
      a)&\ \|\uod - \uob \|_{\hsp}^2 \leq C\delta^{6 - 2s'} (\|\uob -
      \uo \|_{\htr}^2 + \| \uob \|_{\htr}^2),
      \\
      b)&\ \|\uod - \uob\|_{\hspru}^2 \leq C\delta^{4 - 2s'} (\|\uob
      - \uo \|_{\htr}^2 + \| \uob \|_{\htr}^2). 
    \end{aligned}
  \end{equation*}
  Consequently
  \begin{equation*}
    \begin{aligned}
      \|\omega \|_{\lihsp}^2 \leq \big( \delta^{4 - 2s'}(\delta^2 +
      C\alpha^2) (\|\uob - \uo \|_{\htr}^2 + \| \uob \|_{\htr}^2) +
      C\nu T \big)C(T).
    \end{aligned}
  \end{equation*}
  We use the above inequality and the bound $ \| \nabla \uad \|_{\htr}
  \leq C \| \ua \|_{\htr}/\delta \leq C/\delta$ for the solutions to
  the Euler-Voigt equations (the proof is similar to that
  of~\eqref{eq:NSV63u}). Inserting in relation~\eqref{eq:NSV85} gives,
  after some manipulations,
  \begin{equation*}
    \begin{aligned}
      \frac{d}{dt}( \|\omega\|_{\htr}^2 + \alpha^2 \| \nabla
      \omega\|_{\htr}^2 ) \leq C ( \|\uad\|_{\htr} + \|\uabn\|_{\htr}
      +1) \| \omega \|_{\htr}^2 + C\frac{\nu}{\delta^2} + \widehat
      C^2(T, \alpha, \delta, \nu)
      \\
      \leq C(\alpha, \beta, \delta, \nu) \big(\| \omega \|_{\htr}^2 +
      \alpha^2 \|\nabla \omega \|_{\htr}^2\big) +
      C\frac{\nu}{\delta^2} + \widehat C^2(T, \alpha, \delta, \nu).
    \end{aligned}
  \end{equation*} 
  Where ${C(\alpha, \beta, \delta, \nu):=
  \|\uad\|_{\htr} + \|\uabn\|_{\htr} +1}$ and
  $\widehat C(T, \alpha, \delta, \nu) := \big( \delta^{2 -
    2s'}(\delta^2 + C\alpha^2) (\|\uob - \uo \|_{\htr}^2 + \| \uob
  \|_{\htr}^2)+ \frac{\nu}{\delta^2} T
  \big)^{\frac{1}{2}}C^{\frac{1}{2}}(T)$. Then, using the Gronwall
  lemma, we deduce that
  \begin{equation*}
    \|\omega \|_{L^\infty(0,T;\htr)}^2 \leq  \Big( \big(\|\omega_0\|_{\htr}^2
    +  \alpha^2 \|\nabla \omega_0\|_{\htr}^2 \big) +  \big(
    C\frac{\nu}{\delta^2} +  \widehat C^2(T, \alpha, \delta,
    \nu) \big)T \Big)e^{C(T, \alpha, \beta, \delta, \nu)},
  \end{equation*}
  { with $C(T, \alpha, \beta, \delta, \nu) := C\int_0^T ( \|\uad
    (s)\|_{\htr} + \|\uabn (s)\|_{\htr} + 1)ds$.}

  \noindent To conclude, we need a further estimate of $\alpha^2 \|\nabla
  \omega_0\|_{\htr}^2 =\alpha^2 \|\nabla (\uod - \uob)\|_{\htr}^2$.
  Consider the regularized initial data $\uobd$. It follows that
  \begin{equation}
    \label{eq:NSV85xy}
    \begin{aligned}
      \alpha^2 \|\nabla (\uod - \uob)\|_{\htr}^2 &\leq C \alpha^2
      \|\uod - \uobd\|_{\hq}^2 + C\alpha^2 \|\uobd - \uob\|_{\hq}^2
      \\
      &\leq C \frac{\alpha^2}{\delta^2} \|\uob - \uo\|_{\htr}^2 +
      C\alpha^2 \|\uob\|_{\hq}^2,
    \end{aligned}
  \end{equation}
  and noting that $ \|\uod - \uob\|_{\htr}^2 \leq
  C\|\uod-\uo\|_{\htr}^2 + C\|\uob - \uo\|_{\htr}^2$, we obtain
  \begin{equation}
    \label{eq:NSV85yy}
    \begin{aligned}
      \|\omega \|_{\lihtr}^2 \leq & C\Big( \|\uod - \uo \|_{\htr}^2 +
      C(1 +\frac{\alpha^2}{\delta^2}) \|\uob - \uo\|_{\htr}^2
      \\
      & + C\alpha^2 \|\uob\|_{\hq}^2 + \big( C\frac{\nu}{\delta^2} +
      \widehat C^2(T, \alpha, \delta, \nu) \big)T \Big)e^{C(T,\alpha,
        \beta, \delta, \nu)} .
    \end{aligned}
  \end{equation}

  \bigskip

  \noindent\textit{Estimate for }$\| \uad - \uab\|_{\lihtr}$: Taking
  the $\htr$-energy estimate for $\omega^{\alpha,
    \beta}_\delta:=\uad-\uab$ (also in this case, we drop $\alpha,
  \beta$ and $\delta$), we arrive at
  \begin{equation*}
    \begin{aligned}
      \frac{d}{dt}( \|\omega\|_{\htr}^2 + \alpha^2 \| \nabla
      \omega\|_{\htr}^2) \leq ( \|\uad\|_{\htr} + \|\uab\|_{\htr}) \|
      \omega \|_{\htr}^2 + \|\uad\|_{\hq} \| \omega \|_{L^\infty} \|
      \omega \|_{\htr}.
    \end{aligned}
  \end{equation*}
%
  We estimate the term $\| \omega \|_{L^\infty (\T^3)}$ in the usual
  way and we find
  \begin{equation*}
    \begin{aligned}
      \frac{d}{dt}( \|\omega \|_{\hsp}^2 + \alpha^2\|\nabla \omega
      \|_{\hsp}^2) & \leq (\|\uad \|_{\hspru} +
      \|\uab\|_{\hspru}) \| \omega \|^2_{\hsp}
      \\
      & \leq (\|\uad \|_{\hspru} + \|\uab\|_{\hspru})
      ( \|\omega \|_{\hsp}^2 + \alpha^2\|\nabla \omega
      \|_{\hsp}^2)
      \\
      & \leq C( \|\omega \|_{\hsp}^2 + \alpha^2\|\nabla \omega
      \|_{\hsp}^2).
    \end{aligned}
  \end{equation*}
  Then, applying the Gronwall lemma and using the
  bound~\eqref{NSV86a}, we get
  \begin{equation*}
    \begin{aligned}
      \|\omega \|_{\lihsp}^2
      \leq\big( \delta^{4 - 2s'}(\delta^2 + \alpha^2 ) (\| \uob - \uo
      \|_{\htr}^2 + \| \uob \|_{\htr}^2) \big) C(T).
    \end{aligned}
  \end{equation*}
  making  computations similar to those performed
  in~\eqref{eq:NSV76}-\eqref{eq:NSV78}, we finally obtain the estimate
  \begin{equation*}
    \begin{aligned}
      \frac{d}{dt}( \|\omega\|_{\htr}^2 + \alpha^2 \| \nabla
      \omega\|_{\htr}^2) \leq & C(\|\uad\|_{\htr} + \|\uab\|_{\htr} +
      1) ( \| \omega \|_{\htr}^2 + \alpha^2\|\nabla \omega
      \|_{\htr}^2)
      \\
      & + \widehat C^2(T, \alpha, \beta, \delta ) ,
    \end{aligned}
  \end{equation*}
  where $ \widehat C (T, \alpha, \beta, \delta) := \big( \delta^{2 -
    2s'} (\delta^2 + \alpha^2 )(\| \uob - \uo \|_{\htr}^2 + \| \uob
  \|_{\htr}^2) \big)^{\frac{1}{2}} C^{\frac{1}{2}}(T). $ Hence, we get
  \begin{equation*}
    \begin{aligned}
      \|\uad - \uab \|_{\lihtr}^2 \leq & \big( \| \uod -
      \uob\|_{\htr}^2 + \alpha^2\| \nabla ( \uod - \uob)\|_{\htr}^2
      \\
      &+ \widehat{C}{}^2(T,\alpha, \beta, \delta) T \big) e^{C\int_0^T
        ( \|\uad (s)\|_{\htr} + \|\uab (s)\|_{\htr} + 1)ds}.
    \end{aligned}
  \end{equation*}
  Now, arguing as in~\eqref{eq:NSV85xy}, 
  we can conclude that
  \begin{equation}
    \label{eq:NSV87}
    \begin{aligned}
      \| \uad - \uab \|_{\lihtr}^2 \leq & C\Big( \| \uod - \uo
      \|_{\htr}^2 + C(1 + \frac{\alpha^2}{\delta^2}) \| \uob - \uo
      \|_{\htr}^2
      \\
      & + \alpha^2 \| \uob \|_{\hq}^2 +\widehat{C}{}^2(T,\alpha,
      \beta, \delta) T \Big) e^{C(T, \alpha, \beta, \delta)},
    \end{aligned}
  \end{equation}
  { where $C(T, \alpha, \beta, \delta) := C\int_0^T ( \|\uad
    (s)\|_{\htr} + \|\uab (s)\|_{\htr} + 1)ds$.}

  \bigskip
  \noindent\textit{Estimate for }$\| \uab - u \|_{\lihtr}$: Note that,
  up to a sub-sequence $\alpha_n \to 0$ as $n \to \infty$, the needed
  estimate on $\|u^{\alpha_n, \beta_n} - u\|_{\lihtr}$ is provided by
  Theorem~\ref{thm:theorem2}. \par
  As a consequence of the above bound and the
  estimates~\eqref{eq:NSV85yy} and~\eqref{eq:NSV87}, for any positive
  sequence $\{\nu_n\}$, with $\nu_n \to 0$ as $n \to \infty$, letting
  $\beta_n \to 0$, we can choose a pair of sequences $\alpha_n,
  \delta_n \to 0$ (look at the proof of Theorem~\ref{thm:theorem2}),
  such that $\|u_{0, \delta_n} - \uo\|_{\htr}$,
  $\alpha_n^2 \|u^{\beta_n}_{0}\|_{\hq}^2$,
  $\frac{\alpha_n}{\delta_n}$, $\frac{\alpha_n}{\delta_n^{s' -1}}$,
  and $\frac{\nu_n}{\delta_n^2}$ go to zero as $n$ goes to infinity. %
  Hence, we get
  \begin{equation*}
    \|u^{\alpha_n, \beta_n, \nu_n} -u\|_{\lihtr} \to 0\quad\text{as }n
    \to \infty, 
  \end{equation*}
  and the thesis follows.
\end{proof}
\begin{remark}
  The  result holds also in spaces of more regular functions
  $H_m$, $m\geq3$, by using essentially the same techniques.
\end{remark}

\def\ocirc#1{\ifmmode\setbox0=\hbox{$#1$}\dimen0=\ht0 \advance\dimen0
  by1pt\rlap{\hbox to\wd0{\hss\raise\dimen0
  \hbox{\hskip.2em$\scriptscriptstyle\circ$}\hss}}#1\else {\accent"17 #1}\fi}
  \def\polhk#1{\setbox0=\hbox{#1}{\ooalign{\hidewidth
  \lower1.5ex\hbox{`}\hidewidth\crcr\unhbox0}}} \def\cprime{$'$}


\begin{thebibliography}{10}

\bibitem{Bei1988}
H.~{Beir{\~a}o da Veiga}.
\newblock Kato's perturbation theory and well-posedness for the {E}uler
  equations in bounded domains.
\newblock {\em Arch. Rational Mech. Anal.}, 104(4):367--382, 1988.

\bibitem{Bei1993a}
H.~{Beir{\~a}o da Veiga}.
\newblock Perturbation theorems for linear hyperbolic mixed problems and
  applications to the compressible {E}uler equations.
\newblock {\em Comm. Pure Appl. Math.}, 46(2):221--259, 1993.

\bibitem{Bei1994a}
H.~{Beir{\~a}o da Veiga}.
\newblock Singular limits in compressible fluid dynamics.
\newblock {\em Arch. Rational Mech. Anal.}, 128(4):313--327, 1994.

\bibitem{Bei2009}
H.~{Beir{\~a}o da Veiga}.
\newblock A review on some contributions to perturbation theory, singular
  limits and well-posedness.
\newblock {\em J. Math. Anal. Appl.}, 352(1):271--292, 2009.

\bibitem{Bei2010a}
H.~{Beir{\~a}o da Veiga}.
\newblock On the sharp vanishing viscosity limit of viscous incompressible
  fluid flows.
\newblock In {\em New Directions in Mathematical Fluid Mechanics}, Adv. Math.
  Fluid Mech., pages 113--122. Birkh\"auser, Basel, 2010.

\bibitem{BBM1972}
T.~B. Benjamin, J.~L. Bona, and J.~J. Mahony.
\newblock Model equations for long waves in nonlinear dispersive systems.
\newblock {\em Philos. Trans. Roy. Soc. London Ser. A}, 272(1220):47--78, 1972.

\bibitem{BIL2006}
L.~C. Berselli, T.~Iliescu, and W.~J. Layton.
\newblock {\em Mathematics of {L}arge {E}ddy {S}imulation of turbulent flows}.
\newblock Scientific Computation. Springer-Verlag, Berlin, 2006.

\bibitem{BL2009}
L.~C. Berselli and R.~Lewandowski.
\newblock Convergence of {ADM} models to {N}avier-{S}tokes equations.
\newblock arXiv0912.4121v1, 2009.

\bibitem{CHT2005}
C.~Cao, D.~D. Holm, and E.~S. Titi.
\newblock On the {C}lark-{$\alpha$} model of turbulence: global regularity and
  long-time dynamics.
\newblock {\em J. Turbul.}, 6:Paper 20, 11 pp. (electronic), 2005.

\bibitem{CLT2006}
Y.~Cao, E.~M. Lunasin, and E.~S. Titi.
\newblock Global well-posedness of the three-dimensional viscous and inviscid
  simplified {B}ardina turbulence models.
\newblock {\em Commun. Math. Sci.}, 4(4):823--848, 2006.

\bibitem{CHOT2005}
A.~Cheskidov, D.~D. Holm, E.~Olson, and E.~S. Titi.
\newblock On a {L}eray-{$\alpha$} model of turbulence.
\newblock {\em Proc. R. Soc. Lond. Ser. A Math. Phys. Eng. Sci.},
  461(2055):629--649, 2005.

\bibitem{CF1988}
P.~Constantin and C.~Foias.
\newblock {\em Navier-{S}tokes equations}.
\newblock Chicago Lectures in Mathematics. University of Chicago Press,
  Chicago, IL, 1988.

\bibitem{DE2006}
A.~Dunca and Y.~Epshteyn.
\newblock On the {S}tolz-{A}dams deconvolution model for the large-eddy
  simulation of turbulent flows.
\newblock {\em SIAM J. Math. Anal.}, 37(6):1890--1902 (electronic), 2006.

\bibitem{FHT2002}
C.~Foias, D.~D. Holm, and E.~S. Titi.
\newblock The three dimensional viscous {C}amassa-{H}olm equations, and their
  relation to the {N}avier-{S}tokes equations and turbulence theory.
\newblock {\em J. Dynam. Differential Equations}, 14(1):1--35, 2002.

\bibitem{GL2000}
G.~P. Galdi and W.~J. Layton.
\newblock Approximation of the larger eddies in fluid motions. {I}{I}. {A}
  model for space-filtered flow.
\newblock {\em Math. Models Methods Appl. Sci.}, 10(3):343--350, 2000.

\bibitem{ILT2006}
A.~A. Ilyin, E.~M. Lunasin, and E.~S. Titi.
\newblock A modified-{L}eray-{$\alpha$} subgrid scale model of turbulence.
\newblock {\em Nonlinearity}, 19(4):879--897, 2006.

\bibitem{KLT2009}
V.~K. Kalantarov, Boris Levant, and E.~S. Titi.
\newblock Gevrey regularity for the attractor of the 3{D}
  {N}avier-{S}toke-{V}oight equations.
\newblock {\em J. Nonlinear Sci.}, 19(2):133--152, 2009.

\bibitem{KT2009}
V.~K. Kalantarov and E.~S. Titi.
\newblock Global attractors and determining modes for the 3{D}
  {N}avier-{S}tokes-{V}oight equations.
\newblock {\em Chin. Ann. Math. Ser. B}, 30(6):697--714, 2009.

\bibitem{Kat1972}
T.~Kato.
\newblock Nonstationary flows of viscous and ideal fluids in {${\bf R}^{3}$}.
\newblock {\em J. Functional Analysis}, 9:296--305, 1972.

\bibitem{KP1988}
T.~Kato and G.~Ponce.
\newblock Commutator estimates and the {E}uler and {N}avier-{S}tokes equations.
\newblock {\em Comm. Pure Appl. Math.}, 41(7):891--907, 1988.

\bibitem{KPV1991}
C.~E. Kenig, G.~Ponce, and L.~Vega.
\newblock Well-posedness of the initial value problem for the {K}orteweg-de
  {V}ries equation.
\newblock {\em J. Amer. Math. Soc.}, 4(2):323--347, 1991.

\bibitem{KT2008}
B.~Khouider and E.~S. Titi.
\newblock An inviscid regularization for the surface quasi-geostrophic
  equation.
\newblock {\em Comm. Pure Appl. Math.}, 61(10):1331--1346, 2008.

\bibitem{LT2010}
A.~Larios and Titi~E. S.
\newblock On the higher-order global regularity of the inviscid
  {V}oigth-regularization of the three-dimensional hydrodynamic models.
\newblock {\em Discrete Contin. Dyn. Syst. Ser. B}, 14:603--627, 2010.

\bibitem{LMNR2010}
W.~Layton, C.~C. Manica, M.~Neda, and L.~G. Rebholz.
\newblock Numerical analysis and computational comparisons of the {NS}-alpha
  and {NS}-omega regularizations.
\newblock {\em Comput. Methods Appl. Mech. Engrg.}, 199(13--16):916--931, 2010.

\bibitem{LL2006a}
W.~J. Layton and R.~Lewandowski.
\newblock On a well-posed turbulence model.
\newblock {\em Discrete Contin. Dyn. Syst. Ser. B}, 6(1):111--128 (electronic),
  2006.

\bibitem{LN2007b}
W.~J. Layton and M.~Neda.
\newblock Truncation of scales by time relaxation.
\newblock {\em J. Math. Anal. Appl.}, 325(2):788--807, 2007.

\bibitem{LST2008}
W.~J. Layton, I.~Stanculesu, and C.~Trenchea.
\newblock Theory of the {NS}-$\overline{\omega}$ model: {A} complement to the
  {NS}-$\alpha$ model.
\newblock Technical report, Depth. Math. Pittsburgh Univ., 2008.

\bibitem{Ler1934}
J.~Leray.
\newblock Sur le mouvement d'un liquide visqueux emplissant l'espace.
\newblock {\em Acta Math.}, 63(1):193--248, 1934.

\bibitem{LiT2010}
J.~S. Linshiz and E.~S. Titi.
\newblock On the convergence rate of the {E}uler-{$\alpha$}, an inviscid
  second-grade complex fluid, model to the {E}uler equations.
\newblock {\em J. Stat. Phys.}, 138(1-3):305--332, 2010.

\bibitem{Lio1969}
J.-L. Lions.
\newblock {\em Quelques m\'ethodes de r\'esolution des probl\`emes aux limites
  non lin\'eaires}.
\newblock Dunod, Gauthier-Villars, Paris, 1969.

\bibitem{MB2002}
A.~J. Majda and A.~L. Bertozzi.
\newblock {\em Vorticity and incompressible flow}, volume~27 of {\em Cambridge
  Texts in Applied Mathematics}.
\newblock Cambridge University Press, Cambridge, 2002.

\bibitem{Mas2007}
N.~Masmoudi.
\newblock Remarks about the inviscid limit of the {N}avier-{S}tokes system.
\newblock {\em Comm. Math. Phys.}, 270(3):777--788, 2007.

\bibitem{OT2007}
E.~Olson and E.~S. Titi.
\newblock Viscosity versus vorticity stretching: global well-posedness for a
  family of {N}avier--{S}tokes-alpha-like models.
\newblock {\em Nonlinear Anal.}, 66(11):2427--2458, 2007.

\bibitem{Osk1973}
A.~P. Oskolkov.
\newblock The uniqueness and solvability in the large of boundary value
  problems for the equations of motion of aqueous solutions of polymers.
\newblock {\em Zap. Nau\v cn. Sem. Leningrad. Otdel. Mat. Inst. Steklov.
  (LOMI)}, 38:98--136, 1973.
\newblock Boundary value problems of mathematical physics and related questions
  in the theory of functions, 7.

\bibitem{Osk1982}
A.~P. Oskolkov.
\newblock On the theory of unsteady flows of {K}elvin-{V}oigt fluids.
\newblock {\em Zap. Nauchn. Sem. Leningrad. Otdel. Mat. Inst. Steklov. (LOMI)},
  115:191--202, 310, 1982.
\newblock Boundary value problems of mathematical physics and related questions
  in the theory of functions, 14.

\bibitem{Pro1960a}
G.~Prodi.
\newblock Qualche risultato riguardo alle equazioni di {N}avier-{S}tokes nel
  caso bidimensionale.
\newblock {\em Rend. Sem. Mat. Univ. Padova}, 30:1--15, 1960.

\bibitem{Reb2008}
L.~G. Rebholz.
\newblock A family of new, high order {NS}-$\alpha$ models arising from
  helicity correction in {L}eray turbulence models.
\newblock {\em J. Math. Anal. Appl.}, 342(1):246--254, 2008.

\bibitem{Tay1991}
M.~E. Taylor.
\newblock {\em Pseudodifferential operators and nonlinear {PDE}}, volume 100 of
  {\em Progress in Mathematics}.
\newblock Birkh\"auser Boston Inc., Boston, MA, 1991.

\end{thebibliography}

\end{document}